\documentclass[12pt]{article}%
\usepackage{amsfonts}
\usepackage{amssymb}
\usepackage{amsmath}%
\usepackage{graphicx}
\providecommand{\U}[1]{\protect\rule{.1in}{.1in}}
\newtheorem{theorem}{Theorem}

\newtheorem{corollary}[theorem]{Corollary}

\newtheorem{lemma}[theorem]{Lemma}

\newenvironment{proof}[1][Proof]{\noindent\textbf{#1.} }{{\hfill $\Box$ \\}}
\begin{document}

\title{Understanding Wall's theorem on dependence of Lie relators in Burnside groups}
\author{Michael Vaughan-Lee}
\date{February 2020}
\maketitle

\begin{abstract}
G.E. Wall [7,8] gave two different proofs of a remarkable result about the
multilinear Lie relators satisfied by groups of prime power exponent $q$. He
showed that if $q$ is a power of the prime $p$, and if $f$ is a multilinear
Lie relator in $n$ variables where $n\neq1\operatorname{mod}(p-1)$, then $f=0$
is a consequence of multilinear Lie relators in fewer than $n$ variables. For
years I have struggled to understand his proofs, and while I still have not
the slightest clue about his proof in [7], I finally have some understanding
of his proof in [8]. In this note I offer my insights into Wall's second proof
of this theorem.

\end{abstract}

\section{Introduction}

This note is concerned with the multilinear Lie relators which hold in the
associated Lie rings of groups of prime-power exponent. I refer the reader to
Chapter 2 of my book \cite{vlee93b} for the definition of the associated Lie
ring of a group, and the definition of a Lie relator (or Lie identity). In
that chapter, for every prime power $q=p^{k}$ I explicitly construct a
sequence of Lie elements $K_{n}$ ($n=1,2,\ldots$) where $K_{n}$ is multilinear
in $x_{1},x_{2},\ldots,x_{n}$. (We can think of $K_{n}$ as an element of the
free Lie ring on generators $x_{1},x_{2},\ldots,x_{n}$. It is a linear
combination of Lie products $[y_{1},y_{2},\ldots,y_{n}]$ where $y_{1}%
,y_{2},\ldots y_{n}$ is a permutation of $x_{1},x_{2},\ldots,x_{n}$. The
element $K_{1}$ equals $qx_{1}.$) I prove that the associated Lie ring of any
group of exponent $q$ satisfies the identity $K_{n}=0$ for all $n$. I also
prove that if $f=0$ is a multilinear identity which holds in the associated
Lie ring of every group of exponent $q$, then $f=0$ is a consequence of the
identities $K_{n}=0$ ($n=1,2,\ldots$).

G.E. Wall \cite{wall86}, \cite{wall90} gave two proofs of the following
remarkable theorem which shows that in some sense most of the identities
$K_{n}=0$ are redundant.

\begin{theorem}
If $n\neq1\operatorname{mod}(p-1)$ then the identity $K_{n}=0$ is a
consequence of the identities $K_{m}=0$ for $m<n$.
\end{theorem}

This result ties in with (but vastly extends) previously known results about
groups of exponent $p$. It has been known for many years that the associated
Lie rings of groups of exponent $p$ have characteristic $p$ and satisfy the
$(p-1)$-Engel identity. These identities have helped enormously in the study
of finite groups of exponent $p$. The identity $px=0$ is the identity $K_{1}=0$ 
(in exponent $p$), and the $(p-1)$-Engel identity is equivalent to $K_{p}=0$ in
characteristic $p$. Magnus \cite{magnus50} proved that all Lie identities of
weight at most $p-1$ that hold in the associated Lie rings of groups of exponent
$p$ are consequences of the identity $px=0$, and Sanov \cite{sanov52}
extended this result to show that all Lie identities of weight at most $2p-2$
are consequences of the identity $px=0$ and the $(p-1)$-Engel identity. On the
other hand, Wall \cite{wall67} found a multilinear identity of weight $2p-1$
which holds in the associated Lie rings of groups of exponent $p$, and he
showed that for $p=5,7,11$ this identity is not a consequence of the
$(p-1)$-Engel identity in characteristic $p$. Havas and Vaughan-Lee
\cite{havasvl09} extended this result to the primes $13,17,19$. (It follows
that $K_{2p-1}=0$ is \emph{not} a consequence of the identities $\{K_{n}%
=0\,|\,n<2p-1\}$ for $p=5,7,11,13,17,19$.)

Most of Wall's proof in \cite{wall90} involves calculations in the free associative
algebra with unity over the rationals $\mathbb{Q}$, with free generators
$x_{1},x_{2},\ldots$. Wall calls this algebra $A$, and turns $A$ into a Lie
algebra over $\mathbb{Q}$ by setting $[a,b]=ab-ba$. He lets $L$ be the Lie
subalgebra of $A$ generated by $x_{1},x_{2},\ldots$. It is well known that $L$
is a free Lie algebra, freely generated by $x_{1},x_{2},\ldots$. For a proof
of this fact see [5, Corollary 1.4.2]. Wall then lets $R$ be the set of all
linear mappings $\theta:A\rightarrow A$ which commute with all algebra
endomorphisms $\varepsilon:A\rightarrow A$ such that $\varepsilon(L)\subseteq
L$. Wall implicitly assumes a number of properties of $R$, and states other
important properties without proof. No doubt these properties are
straightforward enough (even elementary) for those who are familiar with $R$,
and in fact their proofs are not hard. But I had never come across $R$ before,
and I could not begin to understand Wall's proof of Theorem 1 until I had
found my own proofs of these properties. So in the section below I develop the
properties of $R$ which Wall needs. I give the definitions of the relators
$K_{n}$ ($n=1,2,\ldots$) in Section 3, and establish some of their properties.
Then in the final section I give his proof of Theorem 1. The proof itself is
quite short, but it takes quite a lot of work in Section 2 and Section 3 to set
up the machinery needed.

We end the introduction stating and proving a standard result about products
of the free generators of $A$. The proof is easy enough, but we use the result
below (several times) and so it is convenient to state it as a lemma.

\begin{lemma}
Let $U_{r}\subset A$ be the set of all sums $\sum_{j}m_{j}$ where the summands
$m_{j}$ are products
\[
y_{1}y_{2}\ldots y_{i-1}[y_{i},y_{i+1}]y_{i+2}\ldots y_{r}
\]
where $1\leq i<r$ and where $y_{1},y_{2},\ldots,y_{r}$ is a permutation of
$x_{1},x_{2},\ldots,x_{r}$. Let $\pi$ be a permutation in the symmetric group
$S_{r}$. Then $x_{\pi1}x_{\pi2}\ldots x_{\pi r}=x_{1}x_{2}\ldots x_{r}+u$ for
some $u\in U_{r}$.
\end{lemma}

\begin{proof}
We apply a \textquotedblleft bubble sort\textquotedblright\ to the product
$x_{\pi1}x_{\pi2}\ldots x_{\pi r}$. If $\pi=1_{S_{r}}$ then there is nothing
to prove. If $\pi\neq1_{S_{r}}$, look for the first integer $i$ such that $\pi
i>\pi(i+1)$. Then%
\begin{align*}
& x_{\pi1}\ldots x_{\pi i}x_{\pi(i+1)}\ldots x_{\pi r}\\
& =x_{\pi1}\ldots x_{\pi(i+1)}x_{\pi i}\ldots x_{\pi r}+x_{\pi1}\ldots\lbrack
x_{\pi i},x_{\pi(i+1)}]\ldots x_{\pi r}\\
& =x_{\pi1}\ldots x_{\pi(i+1)}x_{\pi i}\ldots x_{\pi r}+u\text{ with }u\in
U_{r}\text{.}%
\end{align*}
We replace $x_{\pi1}\ldots x_{\pi i}x_{\pi(i+1)}\ldots x_{\pi r}$ by $x_{\pi
1}\ldots x_{\pi(i+1)}x_{\pi i}\ldots x_{\pi r}$, and iterate. After a number
of iterations we obtain $x_{1}x_{2}\ldots x_{r}$.
\end{proof}

\section{The algebra $R$}

As mentioned above, Wall defines $R$ to be the set of all linear mappings
$\theta:A\rightarrow A$ which commute with all algebra endomorphisms
$\varepsilon:A\rightarrow A$ such that $\varepsilon(L)\subseteq L$. It is
clear from the definition of $R$ that it is closed under addition, scalar
multiplication, and composition of mappings. So $R$ is an algebra over
$\mathbb{Q}$.

\begin{lemma}
If $\theta\in R$ then $\theta(1)\in\mathbb{Q}$, and if $r>0$, then
\[
\theta(x_{1}x_{2}\ldots x_{r})=\sum_{\pi\in S_{r}}a_{\pi}x_{\pi 1}%
x_{\pi 2}\ldots x_{\pi r}
\]
for some $a_{\pi}\in\mathbb{Q}$.
\end{lemma}

\begin{proof}
First consider $\theta(1)$. If we let $\varepsilon$ be the endomorphism of $A$
mapping $x_{i}$ to $0$ for all $i$, then the fact that $\varepsilon
\theta=\theta\varepsilon$ implies that $\theta(1)\in\mathbb{Q}$.

Next consider $\theta(x_{1})$. If we let $\varepsilon$ be the endomorphism of
$A$ mapping $x_{1}$ to $x_{1}$ and mapping all other free generators of $A $
to zero, then the fact that $\varepsilon\theta=\theta\varepsilon$ implies that
$\theta(x_{1})$ is a polynomial in $x_{1}$. Let $\theta(x_{1})=f(x_{1})$ where
$f$ is a polynomial. If $k\in\mathbb{Q}$, then $\theta$ commutes with the
endomorphism of $A$ mapping $x_{1}$ to $kx_{1}$ and mapping $x_{i}$ to $x_{i}$
for $i>1$. So $\theta(kx_{1})=f(kx_{1})$. But $\theta$ is linear, and so
$\theta(kx_{1})=kf(x_{1})$ This implies that $f(x_{1})=ax_{1}$ for some
$a\in\mathbb{Q}$.

Similar considerations imply that $\theta(x_{1}x_{2}\ldots x_{r})$ lies in the
subalgebra of $A$ generated by $x_{1},x_{2},\ldots,x_{r}$, and by considering
endomorphisms of $A$ mapping $x_{i}$ to $k_{i}x_{i}$ for $i=1,2,\ldots,r$
($k_{i}\in\mathbb{Q}$) we see that either $\theta(x_{1}x_{2}\ldots x_{r})=0$,
or $\theta(x_{1}x_{2}\ldots x_{r})$ is homogeneous of degree 1 in each of
$x_{1},x_{2},\ldots,x_{r}$.
\end{proof}

\bigskip Note that the fact that $\theta$ commutes with all $L$ preserving
endomorphisms of $A$ implies that if%
\[
\theta(x_{1}x_{2}\ldots x_{r})=\sum_{\pi\in S_{r}}\alpha_{\pi}x_{\pi %
1}x_{\pi 2}\ldots x_{\pi r}
\]
then%
\[
\theta(y_{1}y_{2}\ldots y_{r})=\sum_{\pi\in S_{r}}\alpha_{\pi}y_{\pi %
1}y_{\pi 2}\ldots y_{\pi r}
\]
for all $y_{1},y_{2},\ldots,y_{r}\in L$. So the values of $\theta(x_{1}%
x_{2}\ldots x_{r})$ for $r=0,1,2,\ldots$ determine $\theta$. We let
$\sigma_{0}(\theta)=\theta(1)$, and if%
\[
\theta(x_{1}x_{2}\ldots x_{r})=\sum_{\pi\in S_{r}}\alpha_{\pi}x_{\pi %
1}x_{\pi 2}\ldots x_{\pi r}
\]
then we let $\sigma_{r}(\theta)=\sum_{\pi\in S_{r}}\alpha_{\pi}\pi^{-1}$ in the
group ring $\mathbb{Q}S_{r}$. So $\theta$ is determined by the sequence
$(\sigma_{0}(\theta),\sigma_{1}(\theta),\ldots)$. It is straightforward to
show that if $\theta,\varphi\in R$, and if we define $\theta\circ\varphi$ by
setting
\[
(\theta\circ\varphi)(x_{1}x_{2}\ldots x_{r})=(\theta(\varphi(x_{1}x_{2}\ldots
x_{r}))
\]
then $\sigma_{r}(\theta\circ\varphi)=\sigma_{r}(\theta)\sigma_{r}(\varphi)$
for $r=0,1,2,\ldots$.

Now let $x$ be an indeterminate and let $\mathbb{Q}[[x]]$ be the power series
ring in $x$ over $\mathbb{Q}$. If $w\in\mathbb{Q}[[x]]$ then we define the
linear transformation $\psi_{w}:A\rightarrow A$ as follows. First we set
$\psi_{w}(1)=w(0)$. Next, for $r>0$ we let $\psi_{w}(x_{1}x_{2}\ldots x_{r})$
be the $\{x_{1},x_{2},\ldots,x_{r}\}$-multilinear component of%
\[
w((1+x_{1})(1+x_{2})\ldots(1+x_{r})-1).
\]
Let $y_{1},y_{2},\ldots,y_{r}$ be free generators of $A$. Then if%
\[
\psi_{w}(x_{1}x_{2}\ldots x_{r})=\sum_{\pi\in S_{r}}\alpha_{\pi}x_{\pi %
1}x_{\pi 2}\ldots x_{\pi r}
\]
we set%
\[
\psi_{w}(y_{1}y_{2}\ldots y_{r})=\sum_{\pi\in S_{r}}\alpha_{\pi}y_{\pi %
1}y_{\pi 2}\ldots y_{\pi r}.
\]
Note that if $n>r$ then $((1+x_{1})(1+x_{2})\ldots(1+x_{r})-1)^{n}$ has no
terms of degree $r$, so that $\psi_{w}(x_{1}x_{2}\ldots x_{r})$ is well
defined. In fact, if $w=\sum_{n=0}^{\infty}\beta_{n}x^{n}$ then%
\[
\psi_{w}(x_{1}x_{2}\ldots x_{r})=\psi_{w^{\prime}}(x_{1}x_{2}\ldots x_{r})
\]
where $w^{\prime}=\sum_{n=0}^{r}\beta_{n}x^{n}$. So $\psi_{w}:A\rightarrow A$
is a well defined linear transformation. (There is no problem over
convergence, since if $a\in A$ has degree $r$ then $\psi_{w}(a)=\psi
_{w^{\prime}}(a)$.) In addition, it is clear that if we let $V$ be the
$\mathbb{Q}$-linear span of $x_{1},x_{2},\ldots$, then $\psi_{w}$ commutes
with any endomorphism $\varepsilon:A\rightarrow A$ such that $\varepsilon
(V)\subseteq V$.

\begin{lemma}
If $w\in\mathbb{Q}[[x]]$ then $\psi_{w}\in R$.
\end{lemma}

\begin{proof}
By linearity, it is sufficient to show that $\psi_{x^{m}}\in R$ for
$m=1,2,\ldots$. To this end we introduce some notation to enable us to
describe the form of $\psi_{x^{m}}(x_{1}x_{2}\ldots x_{r})$. If $S=\{i_{1}%
,i_{2},\ldots,i_{k}\}$ is a subset of $\{1,2,\ldots,r\}$ with $i_{1}%
<i_{2}<\ldots<i_{k}$ then we define $x_{S}=x_{i_{1}}x_{i_{2}}\ldots x_{i_{k}}%
$. So%
\[
\psi_{x^{m}}(x_{1}x_{2}\ldots x_{r})=\sum x_{S_{1}}x_{S_{2}}\ldots x_{S_{m}}
\]
where the summation is taken over all partitions of $\{1,2,\ldots,r\}$ into an
ordered sequence of disjoint non-empty subsets $S_{1},S_{2},\ldots,S_{m}$.
Each partition of $\{1,2,\ldots,r\}$ into $m$ disjoint non-empty subsets
yields $m!$ ordered sequences $S_{1},S_{2},\ldots,S_{m}$. If $y_{1}%
,y_{2},\ldots,y_{r}\in A$ we let $U_{r}(y_{1},y_{2},\ldots,y_{r})$ be the
image of $\psi_{x^{m}}(x_{1}x_{2}\ldots x_{r})$ under an endomorphism of $A$
mapping $x_{i}$ to $y_{i}$ for $i=1,2,\ldots,r$. We need to show that if
$y_{1},y_{2},\ldots,y_{r}\in L$ then $\psi_{x^{m}}(y_{1}y_{2}\ldots
y_{r})=U_{r}(y_{1},y_{2},\ldots,y_{r})$, and by linearity it is sufficient to
establish this in the case when the $y_{i}$ are Lie products of the form
$[x_{j_{1}},x_{j_{2}},\ldots,x_{j_{k}}]$. The key to proving this is to show
that if $1\leq i\leq r$ then%
\begin{align*}
& U_{r+1}(x_{1},\ldots,x_{i},x_{i+1},\ldots,x_{r+1})-U_{r+1}(x_{1}%
,\ldots,x_{i+1},x_{i},\ldots,x_{r+1})\\
& =U_{r}(x_{1},\ldots,[x_{i},x_{i+1}],\ldots,x_{r+1}).
\end{align*}
To see this write%
\[
U_{r+1}(x_{1},\ldots,x_{i},x_{i+1},\ldots,x_{r+1})=\sum x_{S_{1}}x_{S_{2}%
}\ldots x_{S_{m}},
\]
where now the sum is over all partitions of $\{1,2,\ldots,r,r+1\}$ into an
ordered sequence of disjoint non-empty subsets $S_{1},S_{2},\ldots,S_{m}$.
Decompose the sum 
\[
\sum x_{S_{1}}x_{S_{2}}\ldots x_{S_{m}}
\]
into $B+C$, where
$B$ is the sum of all $x_{S_{1}}x_{S_{2}}\ldots x_{S_{m}}$ where $i$ and $i+1$
lie in different subsets, and where $C$ is the sum of all $x_{S_{1}}x_{S_{2}%
}\ldots x_{S_{m}}$ where $i$ and $i+1$ lie in the same subset. Note that%
\[
C=U_{r}(x_{1},\ldots,x_{i}x_{i+1},\ldots,x_{r+1}).
\]
Interchanging $x_{i}$ and $x_{i+1}$ leaves $B$ unchanged (although it permutes
the summands) and maps $C$ to $U_{r}(x_{1},\ldots,x_{i+1}x_{i},\ldots
,x_{r+1})$. So%
\begin{align*}
& U_{r+1}(x_{1},\ldots,x_{i},x_{i+1},\ldots,x_{r+1})-U_{r+1}(x_{1}%
,\ldots,x_{i+1},x_{i},\ldots,x_{r+1})\\
& =U_{r}(x_{1},\ldots,x_{i}x_{i+1},\ldots,x_{r+1})-U_{r}(x_{1},\ldots
,x_{i+1}x_{i},\ldots,x_{r+1})\\
& =U_{r}(x_{1},\ldots,[x_{i},x_{i+1}],\ldots,x_{r+1})
\end{align*}
as claimed.

Now consider $\psi_{x^{m}}(y_{1}y_{2}\ldots y_{r})$ where the $y_{i}$ are Lie
products of the free generators of $L$. We need to prove that $\psi_{x^{m}%
}(y_{1}y_{2}\ldots y_{r})=U_{r}(y_{1},y_{2},\ldots,y_{r})$, and we do this by
induction on $\sum_{i=1}^{r}\deg(y_{i})-r$. If $\sum_{i=1}^{r}\deg(y_{i}%
)-r=0$, then all the $y_{i}$ have weight 1 and there is nothing to prove. So
suppose that $1\leq i\leq r$ and that $y_{i}=[z,t]$ for some $z,t\in L$. Then%
\begin{align*}
& \psi_{x^{m}}(y_{1}y_{2}\ldots y_{r})\\
& =\psi_{x^{m}}(y_{1}\ldots\lbrack z,t]\ldots y_{r})\\
& =\psi_{x^{m}}(y_{1}\ldots zt\ldots y_{r})-\psi_{x^{m}}(y_{1}\ldots tz\ldots
y_{r})
\end{align*}
and%
\begin{align*}
& U_{r}(y_{1},y_{2},\ldots,y_{r})\\
& =U_{r}(y_{1},\ldots,[z,t],\ldots,y_{r})\\
& =U_{r+1}(y_{1},\ldots,z,t,\ldots,y_{r})-U_{r+1}(y_{1},\ldots,t,z,\ldots
,y_{r}).
\end{align*}
So it is sufficient to show that%
\[
\psi_{x^{m}}(y_{1}\ldots zt\ldots y_{r})=U_{r+1}(y_{1},\ldots,z,t,\ldots
,y_{r})
\]
and that%
\[
\psi_{x^{m}}(y_{1}\ldots tz\ldots y_{r})=U_{r+1}(y_{1},\ldots,t,z,\ldots
,y_{r})
\]
and this follows by induction.
\end{proof}

\begin{lemma}
If $\psi\in R$ then $\psi=\psi_{w}$ for some unique $w\in\mathbb{Q}[[x]]$.
\end{lemma}

\begin{proof}
Let $\psi\in R$ and let $\psi(1)=\alpha_{0}$, $\psi(x_{1})=\alpha_{1}x_{1}$
with $\alpha_{0},\alpha_{1}\in\mathbb{Q}$. Let $\varphi=\psi-\alpha_{0}%
\psi_{1}-\alpha_{1}\psi_{x}$. Then $\varphi(1)=\varphi(x_{1})=0$. Let
$\varphi(x_{1}x_{2})=\alpha x_{1}x_{2}+\beta x_{2}x_{1}$. Since $\varphi\in R$
it follows that $\varphi([x_{1},x_{2}])=0$ so that $\varphi(x_{1}%
x_{2})=\varphi(x_{2}x_{1})$. Hence $\alpha=\beta$. Let $\alpha
_{2}=\alpha$, and let $\varphi_{2}=\psi-\alpha_{0}\psi_{1}-\alpha_{1}\psi
_{x}-\alpha_{2}\psi_{x^{2}}$. Then $\sigma_{0}(\varphi_{2})=\sigma_{1}%
(\varphi_{2})=\sigma_{2}(\varphi_{2})=0$. Suppose by induction that for some
$n>2$ we have found $\alpha_{0},\alpha_{1},\ldots,\alpha_{n-1}\in\mathbb{Q}$
such that%
\[
\varphi_{n-1}=\psi-\alpha_{0}\psi_{1}-\alpha_{1}\psi_{x}-\ldots-\alpha_{n-1}\psi_{x^{n-1}}
\]
has the property that $\sigma_{i}(\varphi_{n-1})=0$ for $i=0,1,\ldots,n-1$.
Then $\varphi_{n-1}(y_{1}y_{2}\ldots y_{m})=0$ whenever $m<n$ and $y_{1}%
,y_{2},\ldots,y_{m}\in L$. It follows from Lemma 2 that%
\[
\varphi_{n-1}(x_{\pi1}x_{\pi2}\ldots x_{\pi n})=\varphi_{n-1}(x_{1}x_{2}\ldots
x_{n})
\]
for all permutation $\pi\in S_{n}$. So $\sigma_{n}(\varphi_{n-1})=\alpha
_{n}\sum_{\pi\in S_{n}}\pi$ for some $\alpha_{n}\in\mathbb{Q}$, and%
\[
\varphi_{n}=\varphi=\psi-\alpha_{0}\psi_{1}-\alpha_{1}\psi_{x}-\ldots
-\alpha_{n-1}\psi_{x^{n-1}}-\alpha_{n}\psi_{x^{n}}
\]
has the property that $\sigma_{i}(\varphi_{n})=0$ for $i=0,1,\ldots,n$.
Continuing in this way we obtain a sequence of rationals $\alpha_{0}%
,\alpha_{1},\ldots$ such that $\psi=\psi_{w}$ where%
\[
w=\sum_{i=0}^{\infty}\alpha_{i}x^{i}.
\]

\end{proof}

\begin{lemma}
Let $X=x+1$. Then $\psi_{X^{m}}\circ\psi_{X^{n}}=\psi_{X^{mn}}$ for all
$m,n\geq0$.
\end{lemma}

\begin{proof}
It is easy to see that $\sigma_{0}(\psi_{X^{m}})=1$ and
$\sigma_{1}(\psi_{X^{m}})=m1_{S_1}$ for all $m$.
So 
$\sigma_{i}(\psi_{X^{m}}\circ\psi_{X^{n}})=\sigma_{i}(\psi_{X^{mn}}) $ for
$i=0,1$. Suppose by induction that we have shown that 
$\sigma_{i}(\psi_{X^{m}}\circ\psi_{X^{n}})=\sigma_{i}(\psi_{X^{mn}})$ for 
$i=0,1,\ldots,r-1$ for some $r>1$. Let%
\[
\sigma_{r}(\psi_{X^{m}})=\sum_{\pi\in S_{r}}\alpha_{\pi}\pi,\;\sigma_{r}%
(\psi_{X^{n}})=\sum_{\pi\in S_{r}}\beta_{\pi}\pi.
\]
Then the coefficients $\alpha_{\pi},\beta_{\pi}$ are non-negative integers,
and it is not hard to see that%
\[
\sum\alpha_{\pi}=m^{r},\;\sum\beta_{\pi}=n^{r}\text{.}
\]
So%
\[
\sigma_{r}(\psi_{X^{m}}\circ\psi_{X^{n}})=\sigma_{r}(\psi_{X^{m}})\sigma
_{r}(\psi_{X^{n}})=\sum_{\pi\in S_{r}}\gamma_{\pi}\pi
\]
for some non-negative integers $\gamma_{\pi}$ with $\sum\gamma_{\pi}=(mn)^{r}%
$. It follows that%
\[
\sigma_{r}(\psi_{X^{m}}\circ\psi_{X^{n}}-\psi_{X^{mn}})=\sum_{\pi\in S_{r}%
}\delta_{\pi}\pi
\]
for some integers $\delta_{\pi}$ with $\sum\delta_{\pi}=0$. But $\sigma
_{i}(\psi_{X^{m}}\circ\psi_{X^{n}}-\psi_{X^{mn}})=0$ for $i=0,1,\ldots,r-1$,
and as we saw in the proof of Lemma 5 this implies that%
\[
\sigma_{r}(\psi_{X^{m}}\circ\psi_{X^{n}}-\psi_{X^{mn}})=\alpha\sum_{\pi\in
S_{r}}\pi
\]
for some $\alpha\in\mathbb{Q}$. It follows that $\sigma_{r}(\psi_{X^{m}}%
\circ\psi_{X^{n}}-\psi_{X^{mn}})=0$, and so by induction we see that
$\sigma_{i}(\psi_{X^{m}}\circ\psi_{X^{n}})=\sigma_{i}(\psi_{X^{mn}})$ for all
$i\geq0$. So $\psi_{X^{m}}\circ\psi_{X^{n}}=\psi_{X^{mn}}$, as claimed.
\end{proof}

\begin{corollary}
The algebra $R$ is commutative.
\end{corollary}

\begin{proof}
Lemma 6 implies that $\psi_{X^{m}}\circ\psi_{X^{n}}=\psi_{X^{n}}\circ
\psi_{X^{m}}$ for all $m,n\geq0$. We can express $x^{m}$ as a linear
combination of $1,X,X^{2},\ldots,X^{m}$, so that $\psi_{x^{m}}$ is a linear
combination of $\psi_{1},\psi_{X},\ldots,\psi_{X^{m}}$. So $\psi_{x^{m}}%
\circ\psi_{x^{n}}=\psi_{x^{n}}\circ\psi_{x^{m}}$ for all $m,n\geq0$, and $R$
is commutative.
\end{proof}

Now let%
\[
z=\log X=\log(1+x)=\sum_{n=1}^{\infty}(-1)^{n+1}\frac{x^{n}}{n}\in
\mathbb{Q}[[x]].
\]
Then%
\[
X=\text{e}^{z}=\sum_{n=0}^{\infty}\frac{z^{n}}{n!}
\]
and%
\[
X^{m}=\text{e}^{mz}=\sum_{n=0}^{\infty}\frac{m^{n}z^{n}}{n!}.
\]
It follows that%
\[
\psi_{X^{m}}=\sum_{n=0}^{\infty}\varepsilon_{n}m^{n}
\]
where%
\[
\varepsilon_{n}=\psi_{\frac{z^{n}}{n!}}\text{ for }n=0,1,\ldots.
\]
(There is no problem over convergence, since if $a\in A$ has degree $r$ then
$\varepsilon_{n}(a)=0$ for $n>r$.)

\begin{lemma}
The elements $\varepsilon_{n}$ ($n=0,1,2,\ldots$) are orthogonal idempotents.
That is, $\varepsilon_{r}\circ\varepsilon_{s}=0$ if $r\neq s$ and
$\varepsilon_{r}\circ\varepsilon_{r}=\varepsilon_{r}$ for $r,s=0,1,2,\ldots$.
\end{lemma}

\begin{proof}
It is enough to show that if $k\geq0$ then%
\begin{equation}
\sigma_{k}(\varepsilon_{r}\circ\varepsilon_{s})=0\text{ if }r\neq s,
\end{equation}
and%
\begin{equation}
\sigma_{k}(\varepsilon_{r}\circ\varepsilon_{r})=\sigma_{k}(\varepsilon_{r}).
\end{equation}
Now $\sigma_{k}(\varepsilon_{r}\circ\varepsilon_{s})=\sigma_{k}(\varepsilon
_{r})\sigma_{k}(\varepsilon_{s})$ and $\sigma_{k}(\varepsilon_{r})=0$ if
$r>k$. So equations (1) and (2) hold if $r>k$ or $s>k$. As noted above,
$\psi_{X^{m}}=\sum_{r=0}^{\infty}\varepsilon_{r}m^{r}$, and so the equation
$\psi_{X^{m}}\circ\psi_{X^{n}}=\psi_{X^{mn}}$ gives%
\[
\sum_{r,s=0}^{k}\sigma_{k}(\varepsilon_{r})\sigma_{k}(\varepsilon_{s}%
)m^{r}n^{s}=\sum_{r=0}^{k}\sigma_{k}(\varepsilon_{r})(mn)^{r}.
\]
This equation holds for all $m,n$, and so equations (1) and (2) also hold when
$r,s\leq k$.
\end{proof}

We need the following number theoretic result.

\begin{lemma}
Let%
\[
z=\log(1+x)=\sum_{n=1}^{\infty}(-1)^{n+1}\frac{x^{n}}{n}
\]
as above, and let $p$ be a prime. For $0<i<p$ let $t_{i}$ be the power series
in $x$ given by%
\[
t_{i}=\sum_{r=0}^{\infty}\frac{1}{(i+r(p-1))!}z^{i+r(p-1)}.
\]
Then the denominators of the coefficients of $t_{i}$ are coprime to $p$.
\end{lemma}

\begin{proof}
The constant term of $t_{i}$ is 0 for all $i=1,2,\ldots,p-1$. Let $n\geq1$ and
let $\alpha_{i}$ be the coefficient of $x^{n}$ in $t_{i}$ ($0<i<p$). (We fix
$n$ for the moment.) We want to show that $\alpha_{i}\in\mathbb{Z}_{(p)}$, the
ring of rationals with denominators coprime to $p$. For $m=1,2,\ldots$ let
$\beta_{m}$ be the coefficient of $x^{n}$ in $\frac{1}{m!}z^{m}$. Then
$\beta_{m}=0$ for $m>n$, and so if we pick $R$ such that $1+R(p-1)\geq n$ then%
\[
\alpha_{i}=\sum_{r=0}^{R}\beta_{i+r(p-1)}
\]
for $0<i<p$. Let $k$ be a positive integer, so that e$^{kz}=(1+x)^{k}$, and
pick out the coefficient of $x^{n}$ on both sides of the equation%
\[
(1+x)^{k}=\sum_{m=0}^{\infty}\frac{1}{m!}(kz)^{m}.
\]
We obtain the equation%
\[
\binom{k}{n}=\sum_{m=1}^{n}k^{m}\beta_{m}=\sum_{i=1}^{p-1}\sum_{r=0}%
^{R}k^{i+r(p-1)}\beta_{i+r(p-1)}.
\]
Now let $N$ be the highest power of $p$ which divides the denominator of any
of $\beta_{1},\beta_{2},\ldots,\beta_{n}$. By Hensel's lemma we can find an
integer $k$ such that $k+p\mathbb{Z}$ is a primitive element in the finite
field $\mathbb{Z}/p\mathbb{Z}$, and such that $k^{p-1}=1\operatorname{mod}%
p^{N}$. If $r>0$, then $k^{i+r(p-1)}-k^{i}$ is divisible by $p^{N}$ (for any
$i$) so that $k^{i+r(p-1)}\beta_{i+r(p-1)}-k^{i}\beta_{i+r(p-1)}\in
\mathbb{Z}_{(p)}$. It follows that%
\[
\sum_{i=1}^{p-1}k^{i}\alpha_{i}=\sum_{i=1}^{p-1}\sum_{r=0}^{R}k^{i}%
\beta_{i+r(p-1)}\in\mathbb{Z}_{(p)}.
\]
If we let $k_{j}=k^{j}$ for $j=1,2,\ldots,p-1$ then we obtain $p-1$ equations%
\[
\sum_{i=1}^{p-1}k_{j}^{i}\alpha_{i}\in\mathbb{Z}_{(p)},
\]
and since the Vandermonde determinant $%
{\displaystyle\prod\limits_{1\leq i<j\leq p-1}}
(k_{i}-k_{j})$ is coprime to $p$ this implies that $\alpha_{i}\in
\mathbb{Z}_{(p)}$ for $i=1,2,\ldots,p-1$.
\end{proof}

\begin{corollary}
Let%
\[
\eta=\psi_{t_{1}}=\sum_{r=0}^{\infty}\varepsilon_{1+r(p-1)}.
\]
Then $\sigma_{0}(\eta)=0$, $\sigma_{1}(\eta)=1_{S_{1}}$. If $n\geq1$, and if
we write $\sigma_{n}(\eta)=\sum_{\pi\in S_{n}}\alpha_{\pi}\pi$, then the
coefficients $\alpha_{\pi}$ have denominators which are coprime to $p$.
Furthermore, if $n\neq1\operatorname{mod}(p-1)$ then $\sum_{\pi\in S_{n}%
}\alpha_{\pi}=0$.
\end{corollary}

\begin{proof}
The constant term of $t_{1}$ is 0, and the coefficient of $x$ is 1. So
$\sigma_{0}(\eta)=0$ and $\sigma_{1}(\eta)=1_{S_{1}}$. Let $n\geq1$ and let
$\sigma_{n}(\eta)=\sum_{\pi\in S_{n}}\alpha_{\pi}\pi$. Then Lemma 9 implies
that the coefficients $\alpha_{\pi}$ have denominators which are coprime to
$p$. Finally, suppose that $n\neq1\operatorname{mod}(p-1)$. Then
$\varepsilon_{n}\circ\eta=0$ since the elements $\varepsilon_{r}$
($r=0,1,2,\ldots$) are orthogonal idempotents. It is easy to see that
$\sigma_{n}(\varepsilon_{n})=\frac{1}{n!}\sum_{\pi\in S_{n}}\pi$ and so%
\[
\left(  \frac{1}{n!}\sum_{\pi\in S_{n}}\pi\right)  \left(  \sum_{\pi\in S_{n}%
}\alpha_{\pi}\pi\right)  =\sigma_{n}(\varepsilon_{n})\sigma_{n}(\eta
)=\sigma_{n}(\varepsilon_{n}\circ\eta)=0.
\]
This implies that $\sum_{\pi\in S_{n}}\alpha_{\pi}\pi=0$.
\end{proof}

\section{The relators $K_{n}$ ($n=1,2,\ldots$)}

If $a\in L$ and if $x_{i}x_{j}\ldots x_{k}$ is a product of the free
generators of $A$ let $[a\,|\,x_{i}x_{j}\ldots x_{k}]=[a,x_{i},x_{j}%
,\ldots,x_{k}]\in L$. If $b=\sum\beta_{i}m_{i}\in A$ is a linear combination
of products $m_{i}$ then let%
\[
\lbrack a\,|\,b]=\sum\beta_{i}[a\,|\,m_{i}].
\]
Then $K_{1}(x_{1})=qx_{1}$, and if $n>0$ then using Wall's notation we have%
\[
K_{n+1}(x_{1},x_{2},\ldots,x_{n+1})=\sum_{r=2}^{q}\binom{q}{r}[x_{n+1}%
\,|\,\psi_{x^{r-1}}(x_{1}x_{2}\ldots x_{n})].
\]
If $G$ is any group of prime-power exponent $q$ then the associated Lie ring
of $G$ satisfies the relations $K_{n}=0$ ($n=1,2,\ldots$). Furthermore, if
$f=0$ is any multilinear Lie relation satisfied by the associated Lie ring of
every group of exponent $q$, then $f=0$ is a consequence of the identities
$K_{n}=0$ ($n=1,2,\ldots$). This is Theorem 2.4.7 and Theorem 2.5.1 of
\cite{vlee93b}.

Let $L_{\mathbb{Z}}$ be the Lie subring of $L$ generated by $x_{1}%
,x_{2},\ldots$. Then $L_{\mathbb{Z}}$ consists of linear combinations%
\[
\sum\alpha_{i}[x_{i1},x_{i2},\ldots,x_{in_{i}}]
\]
of Lie products of the generators $x_{1},x_{2},\ldots$ where the coefficients
$\alpha_{i}$ are integers. So $K_{n}\in L_{\mathbb{Z}}$ for all $n$. For
$n\geq1$ let $I_{n}$ be the Lie ideal of $L_{\mathbb{Z}}$ generated by
elements $K_{m}(a_{1},a_{2},\ldots,a_{m})$ with $m\leq n$ and $a_{1}%
,a_{2},\ldots,a_{m}\in L_{\mathbb{Z}}$. So Theorem 1 says that if
$n\neq1\operatorname{mod}(p-1)$ then $K_{n}\in I_{n-1}$. Note that $I_{n}$ is
actually spanned by the elements $K_{m}(a_{1},a_{2},\ldots,a_{m})$, since%
\[
\lbrack K_{m}(a_{1},a_{2},\ldots,a_{m}),b]=\sum_{i=1}^{m}K_{m}(a_{1}%
,\ldots,[a_{i},b],\ldots,a_{m}).
\]

We define a linear map $\delta:A\rightarrow L$ by setting $\delta(1)=0$, and
setting%
\[
\delta(x_{i}x_{j}\ldots x_{k})=\left\{
\begin{tabular}
[c]{l}%
$\lbrack x_{i},x_{j},\ldots,x_{k}]$ if $i=1,$\\
$0$ if $i\neq1.$%
\end{tabular}
\right.
\]
The map $\delta$ has the useful property that if $a\in L$ is multilinear in
$x_{1},x_{2},\ldots,x_{r}$ then $\delta(a)=a$. (See [5, page 47].)

\begin{lemma}
$\delta(\psi_{X^{q}}(x_{1}))=\psi_{X^{q}}(x_{1})=qx_{1}=K_{1}(x_{1})$, and if
$n\geq2$ 
\[
\delta(\psi_{X^{q}}(x_{1}x_{2}\ldots x_{n}))=K_{n}(x_{2},\ldots,x_{n},x_{1})\operatorname{mod}I_{n-1}.
\]
\end{lemma}

\begin{proof}
The first claim in the lemma is straighforward, and so we let $n\geq2$.%
\[
\psi_{X^{q}}=\sum_{r=0}^{q}\binom{q}{r}\psi_{x^{r}}
\]
and so%
\[
\psi_{X^{q}}(x_{1}x_{2}\ldots x_{n})=qx_{1}x_{2}\ldots x_{n}+\sum_{r=2}%
^{q}\binom{q}{r}\psi_{x^{r}}(x_{1}x_{2}\ldots x_{n}).
\]
If $r\geq2$, then using the notation introduced in the proof of Lemma 4 we
have%
\[
\psi_{x^{r}}(x_{1}x_{2}\ldots x_{n})=\sum x_{S_{1}}x_{S_{2}}\ldots x_{S_{r}}
\]
where the sum is taken over all partitions of $\{1,2,\ldots,n\}$ into an
ordered sequence of non-empty subsets $S_{1},S_{2},\ldots,S_{r}$. For each
non-empty subset $S\subset\{1,2,\ldots,n\}$ we gather together the summands
where $S_{1}=S$, and obtain%
\[
\psi_{x^{r}}(x_{1}x_{2}\ldots x_{n})=\sum_{S}\left(  \sum_{S_{1}=S}x_{S_{1}%
}x_{S_{2}}\ldots x_{S_{r}}\right)  =\sum_{S}x_{S}\psi_{x^{r-1}}%
(x_{\{1,2,\ldots,n\}\backslash S}).
\]
So%
\begin{align*}
& \delta(\psi_{X^{q}}(x_{1}x_{2}\ldots x_{n}))\\
& =q[x_{1},x_{2},\ldots,x_{n}]+\sum_{S}\left(  \sum_{r=2}^{q}\binom{q}%
{r}[\delta(x_{S})\,|\,\psi_{x^{r-1}}(x_{\{1,2,\ldots,n\}\backslash
S})]\right)  .
\end{align*}
If $S$ is a non-empty subset of $\{1,2,\ldots,n\}$, and if we write
$\{1,2,\ldots,n\}\backslash S=\{i,j,\ldots,k\}$ with $i<j<\ldots<k$ then%
\[
\sum_{r=2}^{q}\binom{q}{r}[\delta(x_{S})\,|\,\psi_{x^{r-1}%
}(x_{\{1,2,\ldots,n\}\backslash S})]=K_{n+1-|S|}(x_{i},x_{j}%
,\ldots,x_{k},\delta(x_{S})).
\]
But $\delta(x_{S})=0$ unless $1\in S$, and $K_{n+1-|S|}(x_{i},x_{j}%
,\ldots,x_{k},\delta(x_{S}))\in I_{n-1}$ if $|S|>1$. Clearly $q[x_{1}%
,x_{2},\ldots,x_{n}]\in I_{n-1}$ and so%
\[
\delta(\psi_{X^{q}}(x_{1}x_{2}\ldots x_{n}))=K_{n}(x_{2},\ldots,x_{n}%
,x_{1})\operatorname{mod}I_{n-1}.
\]

\end{proof}

We now define $\Gamma_{n}$ to be the subring of $A$ generated by elements
$K_{m}(a_{1},a_{2},\ldots,a_{m})$ with $m\leq n$ and $a_{1},a_{2},\ldots
,a_{m}\in L_{\mathbb{Z}}$. So $\Gamma_{n}$ consists of integral linear
combinations of products of elements $K_{m}(a_{1},a_{2},\ldots,a_{m})$ ($m\leq
n$). (Wall's definition of $\Gamma_{n}$ is slightly different from this, in
that he allows coefficients in the ring of rationals with denominators which
are coprime to $p$. Wall also indexes $\Gamma_{n}$ differently.)

\begin{lemma}
$\psi_{X^{q}}(x_{1}x_{2}\ldots x_{n})=K_{n}(x_{2},\ldots,x_{n},x_{1}%
)\operatorname{mod}\Gamma_{n-1}$.
\end{lemma}

\begin{proof}
First note that $\psi_{X^{q}}(x_{1})=qx_{1}=K_{1}(x_{1})$. Next note that
$K_{2}(x_{2},x_{1})=\binom{q}{2}[x_{1},x_{2}]$ and that%
\begin{align*}
&  \psi_{X^{q}}(x_{1}x_{2})\\
&  =\binom{q+1}{2}x_{1}x_{2}+\binom{q}{2}x_{2}x_{1}\\
&  =\binom{q+1}{2}[x_{1},x_{2}]+q^{2}x_{1}x_{2}\\
&  =\binom{q}{2}[x_{1},x_{2}]+q[x_{1},x_{2}]+q^{2}x_{1}x_{2}\\
&  =K_{2}(x_{2},x_{1})\operatorname{mod}\Gamma_{1}%
\end{align*}
since $q^{2}x_{1}x_{2}=(qx_{1})(qx_{2})\in\Gamma_{1}$ and $q[x_{1},x_{2}%
]\in\Gamma_{1}$.

We establish Lemma 12 for general $n$ by induction on $n$. So suppose that
$n>2$ and that%
\[
\psi_{X^{q}}(x_{1}x_{2}\ldots x_{m})=K_{m}(x_{2},\ldots,x_{m},x_{1}%
)\operatorname{mod}\Gamma_{m-1}%
\]
for $m<n$. Note that this implies that $\psi_{X^{q}}(x_{1}x_{2}\ldots
x_{m})\in\Gamma_{m}$ for $m<n$.

We extend $A$ to the ring $\widehat{A}$ of formal power series consisting of
formal sums%
\[
\sum_{r=0}^{\infty}u_{r}%
\]
where $u_{r}$ is a homogeneous element of degree $r$ in $A$. If $a\in
\widehat{A}$ has zero constant term then we define%
\[
\text{e}^{a}=\sum_{r=0}^{\infty}\frac{a^{r}}{r!}%
\]
in the usual way. So e$^{a}$ is a unit in $\widehat{A}$ with inverse e$^{-a}$.
It is well known that the group $F$ generated by e$^{x_{1}}$, e$^{x_{2}}$,
\ldots, e$^{x_{n}}$ is a free group with free generators e$^{x_{1}}$,
e$^{x_{2}}$, \ldots, e$^{x_{n}}$. (See [5, pages 41,42].) If $w\in F$ then%
\[
w=1+u_{1}+u_{2}+\ldots
\]
for some $u_{1},u_{2},\ldots\in A$, where $u_{i}$ is homogeneous of degree $i$
for $i=1,2,\ldots$. We set $u=u_{1}+u_{2}+\ldots$, and set%
\[
z=\log w=\log(1+u)=\sum_{r=1}^{\infty}(-1)^{r+1}\frac{u^{r}}{r},
\]
and then $w=\;$e$^{z}$. If we write%
\[
z=z_{1}+z_{2}+\ldots
\]
where $z_{i}\in A$ is homogeneous of degree $i$ for $i=1,2,\ldots$, then
$z_{i}\in L$ for $i=1,2,\ldots$. This is known as the Baker-Campbell-Hausdorff
formula. (See [5, Theorem 2.5.4].)

Now let $w=($e$^{x_{1}}$e$^{x_{2}}$\ldots e$^{x_{n}})^{q}\in F$. We apply what
Wall calls smoothing to $w$. For $i=1,2,\ldots,n$ we let $\delta_{i}%
:\widehat{A}\rightarrow\widehat{A}$ be the endomorphism given by $\delta
_{i}(x_{i})=0$, $\delta_{i}(x_{j})=x_{j}$ if $j\neq i$. So $\delta_{i}$
induces a homomorphism $\delta_{i}:F\rightarrow F$. Note that $\delta_{i}%
^{2}=\delta_{i}$ and that $\delta_{i}\delta_{j}=\delta_{j}\delta_{i}$ for all
$i,j$.

We set $w_{1}=w.(\delta_{1}w)^{-1}$ so that $\delta_{1}(w_{1})=1$. Then we set
$w_{2}=w_{1}.\delta_{2}(w_{1})^{-1}$ so that $\delta_{i}(w_{2})=1$ for
$i=1,2$. Then we set $w_{3}=w_{2}.\delta_{3}(w_{2})^{-1}$ and so on.
Eventually we obtain an element $w_{n}\in F$ such that $\delta_{i}(w_{n})=1$
for $1\leq i\leq n$. Note that $w_{n}$ is a product of $2^{n}$ elements of the
form $(\delta_{i}\delta_{j}\ldots\delta_{k}(w))^{\pm1}$ with one element for
each subset $\{i,j,\ldots,k\}\subset\{1,2,\ldots,n\}$. The exponent is $+1$ if
$|\{i,j,\ldots,k\}|$ is even, and $-1$ if it is odd. Let $w_{n}=\;$e$^{z}$
where $z=z_{1}+z_{2}+\ldots$, with $z_{r}\in L$ homogeneous of degree $r$ for
$r=1,2,\ldots$. Then $\delta_{i}(z_{r})=0$ for $i=1,2,\ldots ,n$ (for all $r$) and 
we can write $z_{r}$ as a linear combination of Lie products $[x_{j_{1}},x_{j_{2}%
},\ldots,x_{j_{r}}]$ with $\{j_{1},j_{2},\ldots,j_{r}\}=\{1,2,\ldots,n\}$. It
follows that $z_{r}=0$ for $r<n$, and that $z_{n}$ is an $\{x_{1},x_{2}%
,\ldots,x_{n}\}$-multilinear element of $L$.

So $w_{n}=1+z_{n}+u_{n+1}+u_{n+2}+\ldots$ where $u_{i}\in A$ is homogeneous of
degree $i$ for $i=n+1,n+2,\ldots$.

Next consider $w$.%
\begin{align*}
w  & =(e^{x_{1}}e^{x_{2}}\ldots e^{x_{n}})^{q}\\
& =1+\sum\psi_{X^{q}}(x_{i}x_{j}\ldots x_{k})+b
\end{align*}
where the sum is taken over all non empty subsets $\{i,j,\ldots,k\}\subset
\{1,2,\ldots,n\}$, and where $b$ is an infinite sum of terms $\alpha m$ where
$\alpha\in\mathbb{Q}$ and where $m$ is a non-multilinear product of (some of)
the generators $x_{1},x_{2},\ldots,x_{n}$. By induction%
\[
w=1+\psi_{X^{q}}(x_{1}x_{2}\ldots x_{n})+a+b
\]
where $a\in\Gamma_{n-1}$. Similarly, if $\{i,j,\ldots,k\}$ is non-empty,
$\delta_{i}\delta_{j}\ldots\delta_{k}(w)=1+c+d$ where $c\in\Gamma_{n-1}$ and
where $d$ is an infinite sum of terms $\alpha m$ where $\alpha\in\mathbb{Q}$
and where $m$ is a non-multilinear product of the generators $x_{1}%
,x_{2},\ldots,x_{n}$. Clearly $(\delta_{i}\delta_{j}\ldots\delta_{k}(w))^{-1}$
can be expressed in the same form. It follows that if we pick out the
$\{x_{1},x_{2},\ldots,x_{n}\}$-multilinear terms in $w_{n}$ then we obtain%
\[
\psi_{X^{q}}(x_{1}x_{2}\ldots x_{n})+e
\]
for some $e\in\Gamma_{n-1}$, so that%
\[
z_{n}=\psi_{X^{q}}(x_{1}x_{2}\ldots x_{n})\operatorname{mod}\Gamma_{n-1}.
\]
This implies that%
\[
z_{n}=\delta(z_{n})=\delta(\psi_{X^{q}}(x_{1}x_{2}\ldots x_{n}%
))\operatorname{mod}\Gamma_{n-1}.
\]
So%
\[
\psi_{X^{q}}(x_{1}x_{2}\ldots x_{n})=\delta(\psi_{X^{q}}(x_{1}x_{2}\ldots
x_{n}))\operatorname{mod}\Gamma_{n-1}%
\]
and by Lemma 11 this implies that%
\[
\psi_{X^{q}}(x_{1}x_{2}\ldots x_{n})=K_{n}(x_{2},\ldots,x_{n},x_{1}%
)\operatorname{mod}\Gamma_{n-1}.
\]

\end{proof}

\begin{corollary}
$\psi_{X^{q}}(x_{1}x_{2}\ldots x_{n})\in\Gamma_{n}$.
\end{corollary}

\begin{lemma}
If $\pi\in S_{n}$ then%
\[
\psi_{X^{q}}(x_{\pi1}x_{\pi2}\ldots x_{\pi n})-\psi_{X^{q}}(x_{1}x_{2}\ldots
x_{n})\in\Gamma_{n-1}
\]
and
\[
K_{n}(x_{\pi1},x_{\pi2},\ldots,x_{\pi n})-K_{n}(x_{1},x_{2},\ldots,x_{n})\in
I_{n-1}.
\]

\end{lemma}

\begin{proof}
From Lemma 2 we see that%
\[
\psi_{X^{q}}(x_{\pi1}x_{\pi2}\ldots x_{\pi n})-\psi_{X^{q}}(x_{1}x_{2}\ldots
x_{n})
\]
is a sum of terms of the form%
\[
\psi_{X^{q}}(y_{1}y_{2}\ldots y_{i-1}[y_{i},y_{i+1}]y_{i+2}\ldots y_{n})
\]
where $1\leq i<n$ and where $y_{1},y_{2},\ldots,y_{n}$ is a permutation of
$x_{1},x_{2},\ldots,x_{n}$. But%
\[
\psi_{X^{q}}(y_{1}y_{2}\ldots y_{i-1}[y_{i},y_{i+1}]y_{i+2}\ldots y_{n})
\]
is the image of $\psi_{X^{q}}(x_{1}x_{2}\ldots x_{n-1})$ under an endomorphism
$\varphi:A\rightarrow A$ mapping $x_{1},x_{2},\ldots,x_{n-1}$ into
$y_{1},y_{2},\ldots,y_{i-1},[y_{i},y_{i+1}],y_{i+2},\ldots,y_{n}$. We can
assume that $\varphi(L_{\mathbb{Z}})\leq L_{\mathbb{Z}}$, which implies that
$\varphi(\Gamma_{n-1})\subset\Gamma_{n-1}$, and so Corollary 13 implies that
\[
\psi_{X^{q}}(y_{1}y_{2}\ldots y_{i-1}[y_{i},y_{i+1}]y_{i+2}\ldots y_{n}%
)\in\Gamma_{n-1}.
\]
So%
\[
\psi_{X^{q}}(x_{\pi1}x_{\pi2}\ldots x_{\pi n})-\psi_{X^{q}}(x_{1}x_{2}\ldots
x_{n})\in\Gamma_{n-1}
\]
as claimed.

This result, combined with Lemma 12 implies that%
\[
K_{n}(x_{\pi1},x_{\pi2},\ldots,x_{\pi n})-K_{n}(x_{1},x_{2},\ldots,x_{n})=a
\]
for some $a\in\Gamma_{n-1}$, and hence that%
\begin{align*}
& K_{n}(x_{\pi1},x_{\pi2},\ldots,x_{\pi n})-K_{n}(x_{1},x_{2},\ldots,x_{n})\\
& =\delta(K_{n}(x_{\pi1},x_{\pi2},\ldots,x_{\pi n})-K_{n}(x_{1},x_{2}%
,\ldots,x_{n}))\\
& =\delta(a)\\
& \in I_{n-1}%
\end{align*}
since $\delta(\Gamma_{n-1})=I_{n-1}$. (Recall that $I_{n-1}$ is the ideal of
$L_{\mathbb{Z}}$ consisting of integer linear combinations of values
$K_{m}(a_{1},a_{2},\ldots,a_{m})$ with $m<n$ and $a_{1},a_{2},\ldots,a_{m}\in
L_{\mathbb{Z}}$.)
\end{proof}

It is perhaps worth observing that the analysis in this section would have been
simpler if in [5] I had defined $K_{n}(x_{1},x_{2},\ldots,x_{n})$ to be
$\delta(\psi_{X^{q}}(x_{1}x_{2}\ldots x_{n}))$. Of course that would have meant
that the proofs of Theorem 2.4.7 and Theorem 2.5.1 in [5] would have to have
been rewritten. But the proofs would have remained essentially the same.
In fact the proof of Lemma 12 above shows that the associated Lie rings of
groups of exponent $q$ satisfy the identity $z_n=0$. So Lemma 12 (and its
proof) give a proof of Theorem 2.4.7 of [5].

\section{Proof of Theorem 1}

Recall from Section 2 that%
\[
z=\log(1+x)=\sum_{n=1}^{\infty}(-1)^{n+1}\frac{x^{n}}{n}
\]
and that%
\[
\varepsilon_{n}=\psi_{\frac{z^{n}}{n!}}\text{ for }n=0,1,\ldots.
\]
In Corollary 10 we defined%
\[
\eta=\sum_{r=0}^{\infty}\varepsilon_{1+r(p-1)}
\]
and we proved that $\sigma_{0}(\eta)=0$, $\sigma_{1}(\eta)=1_{S_{1}}$, and
that if $n\geq1$, and $\sigma_{n}(\eta)=\sum_{\pi\in S_{n}}\alpha_{\pi}\pi$,
then the coefficients $\alpha_{\pi}$ have denominators which are coprime to
$p$. We also showed that if $n\neq1\operatorname{mod}(p-1)$ then $\sum_{\pi\in
S_{n}}\alpha_{\pi}=0$.

So assume that $n\neq1\operatorname{mod}(p-1)$, and let $k$ be the least 
common multiple of
the denominators of all the coefficients which appear in $\sigma_{m}(\eta)$
for $m=1,2,\ldots,n$. Let $\eta^{\prime}=k\eta$. (Note that $k$ is coprime
to $p$.) Then $\sigma_{n}(\eta^{\prime})=\sum_{\pi\in S_{n}}\beta_{\pi}\pi$
for some integers $\beta_{\pi}$ satisfying $\sum\beta_{\pi}=0$. Lemma 12 
and Lemma 14 imply that%
\[
K_{n}(x_{1},x_{2},\ldots,x_{n})=\psi_{X^{q}}(x_{1}x_{2}\ldots x_{n})+a
\]
for some $a\in\Gamma_{n-1}$. Since $\sigma_{1}(\eta)=1_{S_{1}}$,
\begin{align*}
& \eta^{\prime}(K_{n}(x_{1},x_{2},\ldots,x_{n}))\\
& =kK_{n}(x_{1},x_{2},\ldots,x_{n})\\
& =(\eta^{\prime}\circ\psi_{X^{q}})(x_{1}x_{2}\ldots x_{n})+\eta^{\prime}(a).
\end{align*}

Now $R$ is commutative and so%
\begin{align*}
& (\eta^{\prime}\circ\psi_{X^{q}})(x_{1}x_{2}\ldots x_{n})\\
& =\psi_{X^{q}}(\eta^{\prime}(x_{1}x_{2}\ldots x_{n}))\\
& =\sum_{\pi\in S_{n}}\beta_{\pi}\psi_{X^{q}}(x_{\pi^{-1}1}x_{\pi^{-1}2}\ldots
x_{\pi^{-1}n})\\
& \in\Gamma_{n-1}%
\end{align*}
since $\psi_{X^{q}}(x_{\pi^{-1}1}x_{\pi^{-1}2}\ldots x_{\pi^{-1}n}%
)=\psi_{X^{q}}(x_{1}x_{2}\ldots x_{n})\operatorname{mod}\Gamma_{n-1}$ by Lemma
14, and since $\sum_{\pi\in S_{n}}\beta_{\pi}=0$.

Next we show that $\eta^{\prime}(a)\in\Gamma_{n-1}$. Since $a\in\Gamma_{n-1}$, 
we can express
$a$ as sum of terms of the form $mb$ where $m$ is an integer and where each
$b$ in the sum is a product $c_{1}c_{2}\ldots c_{r}$ where each $c_{i}$ has
the form $K_{s}(a_{1},a_{2},\ldots,a_{s})$ ($s<n$) with $a_{i}\in
L_{\mathbb{Z}}$ for $i=1,2,\ldots,s$. So we need to show that $\eta^{\prime
}(c_{1}c_{2}\ldots c_{r})\in\Gamma_{n-1}$ for each of these products
$c_{1}c_{2}\ldots c_{r}$. We can take $a$ to be homogeneous of degree $n$ and
so $r\leq n$. Our choice of $k$ then implies that $\sigma_{r}(\eta^{\prime
})=\sum_{\pi\in S_{r}}\gamma_{\pi}\pi$ where the coefficients $\gamma_{\pi}$
are integers. It follows that%
\[
\eta^{\prime}(c_{1}c_{2}\ldots c_{r})=\sum_{\pi\in S_{r}}\gamma_{\pi}%
c_{\pi^{-1}1}c_{\pi^{-1}2}\ldots c_{\pi^{-1}r}\in\Gamma_{n-1}.
\]

We have shown that
\[
kK_{n}(x_{1},x_{2},\ldots,x_{n})\in\Gamma_{n-1}.
\]
We also have
\[
qK_{n}(x_{1},x_{2},\ldots,x_{n})\in\Gamma_{n-1}
\]
and since $k$ is coprime to $q$ this implies that
\[
K_{n}(x_{1},x_{2},\ldots,x_{n})\in\Gamma_{n-1}.
\]
This in turn implies that%
\begin{align*}
& K_{n}(x_{1},x_{2},\ldots,x_{n})\\
& =\delta(K_{n}(x_{1},x_{2},\ldots,x_{n}))\\
& \in\delta(\Gamma_{n-1})\\
& =I_{n-1}.%
\end{align*}

\end{document}